\documentclass[11pt, centertag, a4paper]{amsart}

\usepackage{amsmath,amssymb,mathrsfs,amsthm}
\usepackage{hyperref,comment}
\usepackage{url}
\usepackage{array,enumerate}
\usepackage{color} 
\usepackage{arXiv2commands}
\usepackage[all]{xy}
\usepackage{tikz}%
\usetikzlibrary{calc}
\usepackage{bm}%

\usepackage{tgtermes,tgheros}%
\makeatletter
    
    \@addtoreset{equation}{section}
  \makeatother

\title{Equi-dimensionalization via subdivision of simplices}
\author{Wataru Kai}
\address{Mathematical Institute, Tohoku University, Aramaki Aoba 6-3, 980-8578 Sendai, Japan}
\email{kaiw@tohoku.ac.jp}

\date{\today}

\begin{document}

\begin{abstract}
We give an alternative proof of Suslin's equi-dimensionalization moving lemma using a different geometric construction.
The new construction provides better control of the degrees of the polynomials describing the geometric procedure.

The new degree bound can be used to improve an earlier result of Hiroyasu Miyazaki and the present author 
on algebraic cycles {\it with modulus}:
the isomorphism in question is now valid for any fixed divisor, 
without the need to take the limit over thickenings.
\end{abstract}

\maketitle

\tableofcontents

\section{Introduction}\label{sec:introduction}

Let $X$ be an affine scheme of finite type over a field $k$.
Suslin \cite{Suslin2000} showed how to render a given closed set $V\subset X\times \bbA ^n$ 
by pullback along an appropriate $X$-morphism $X\times \bbA ^n \to X\times \bbA ^n$ (compatibly with a prescribed map $X\times Z\to X\times Z$ for a hypersurface $Z\subset \bbA ^n$ if any such data is given)
and make it equi-dimensional over $\bbA ^n\setminus Z$.

By orchestrating this construction to form an endomorphism of the cosimplicial scheme $X\times \Delta ^\bullet $ (where $\Delta ^n$ is the algebraic $n$-simplex \eqref{eq:alg_simplex}),
he showed that Bloch's higher Chow group is isomorphic to its equi-dimensional variant:
\begin{theorem}[{Suslin \cite[Th.~2.1]{Suslin2000}}]
\label{thm:Suslin-quis}
	Let $X$ be an affine scheme of finite type over a field $k$
	and $t\ge 0$ be a non-negative integer.
	Then the inclusion 
	\[
		z^\equi _t (X,\bullet ) \to z_t(X,\bullet )
	\]
	is a quasi-isomorphism.
	(See \S \ref{sec:cycle-complexes} for the definitions of these complexes.)
\end{theorem}

This result was used by
Friedlander--Suslin \cite{FriedlanderSuslin} to prove that Bloch's higher Chow group is isomorphic to a version of motivic cohomology. Voevodsky \cite{VoevodskyMotCohIso} then established that this motivic cohomology is the same as the motivic cohomology they had defined earlier;
see also the beautiful exposition in \cite[Lect.~18A, 19]{MazzaVoevodskyWeibel}.

In these notes, we offer an alternative proof of Theorem \ref{thm:Suslin-quis} based on different constructions.
The motivation for seeking new constructions is twofold.
First, 
the identification of the higher Chow group and the motivic cohomology is such an important result that one could not have enough of different proofs thereof.

Second,
Suslin's constructions do not give nice numerical bounds for the polynomials describing the endomorphisms on $X\times \Delta ^\bullet $.
If we write $X=\Spec A$, an $X$-endomorphism on $X\times \Delta ^n \cong \Spec A[Y_1,\dots ,Y_n]$ corresponds to the choice of $n$ elements $f_1^{(n)},\dots ,f_n^{(n)}$ of $A[Y_1,\dots ,Y_n]$.
If we demand that the endomorphisms form a cosimplicial map,
then Suslin's construction produces $f_k^{(n)}$'s having degrees at least $ n+1$ in the variables $Y_j$.

Also, 
when $X=\bbA ^m = \Spec k[X_1,\dots ,X_m] $ is an affine space (to which case the proof is reduced),
we can consider the degrees of $f_k^{(n)}$ in the variables $X_i$.
These degrees heavily depend on the closed set $V$ and his proof does not imply any uniform bound.

Our constructions are based on the classical idea known as {\it subdivision of simplices}.
This involves $(n+1)!$ endomorphisms of $X\times \Delta ^n$. (E.g.\ a triangle is divided into $6$ small triangles by (barycentric) subdivision.)
They are not part of an endomorphism of the cosimplicial scheme $X\times \Delta ^\bullet $ so do not define an endomorphism of the cycle complex of $X$ as easily as in Suslin's construction. Nonetheless, with an additional simplicial maneuver,
their signed sums do give rise to a map of cycle complexes.
The counterparts to $f_k^{(n)}$'s above will have degree $1$ in the variables $Y_j$ because subdivision is written by affine-linear transformations of $\Delta ^\bullet $.

When $X=\bbA ^m=\Spec k[X_1,\dots ,X_m]$,
it turns out that $f^{(n)}_k$ can be taken to have degrees as low as $n+1$ in the variables $X_i$.
The advance from no uniform bound to this rather explicit one is in part due to an improvement of the proof rather than the construction.
In Suslin's arguments, at the cost of letting $\deg _X(f^{(n)}_k)$ soar, the proof of equi-dimensionality goes in a sense uniformly for all fibers.
\footnote{For the interested reader: in the proof of \cite[Th.~1.1]{Suslin2000}, he considers the fiber of $\Phi \inv (V)\subset \bbA ^m\times \bbA ^n$ over each point $y\in \bbA ^n\setminus Z$ and eventually shows that the {\it infinite part} of the fiber $\Phi\inv (V)_y$ (the closure in $\bbP ^{m}$, intersected with the hyperplane $\bbP ^{m-1}$ at infinity) is contained in a closed subset $W_0$ (independent of $y$!) of dimension $\le t-1$.}
In our treatment, more objects appearing in the discussion 
depend genuinely on the fiber.
 We have to respect their individual personalities and duly respond.
 The dedication is rewarded by the uniform bound of $\deg _X(f^{(n)}_k)$.

We demonstrate that these new bounds of degrees can be useful by showing a {\it modulus} analog of Theorem \ref{thm:Suslin-quis}:
\begin{theorem}[{Theorem \ref{thm:Suslin-modulus}}]
\label{thm:modulus-quis-intro}
	Let $X$ be an affine scheme of finite type over a field $k$
	and $D\subset X$ be an effective Cartier divisor.
	Let $t\ge 0$ be a non-negative integer.
	Then the inclusion 
	\[
		z^\equi _t (X|D,\bullet ) \to z_t(X|D,\bullet )
	\]
	is a quasi-isomorphism.
	(See \S \ref{sec:modulus} for the definitions of these complexes.)
\end{theorem}
Formerly, Hiroyasu Miyazaki and the present author \cite{KaiMiyazaki}
had proven a similar result (copying Suslin's very constructions) but we needed to consider the pro-objects on both sides associated with the infinitesimal thickenings $mD$ ($m\ge 1$).
\footnote{In the earlier work, we treated the cubical version of the higher Chow group with modulus, but the same proof works for the simplicial version as well. The new result, without the need of thickenings, is {\em only available for the simplicial} version because we failed to cook up subdivision of cubes and (importantly) homotopy by affine-linear maps. 
}

\subsection{Recollection of the cycle complexes}\label{sec:cycle-complexes}
Let us briefly recall Bloch's higher Chow complex and its equi-dimensional variant from \cite{Bloch1986} \cite{Suslin2000}. %
Denote by 
\begin{equation}\label{eq:alg_simplex}
	\Delta ^n = \Spec k[t_0,\dots ,t_n]/ (t_0+\dots +t_n -1)
\end{equation} 
the algebraic simplex over the base field $k$.
For each injection $i\colon \{ 0,1,\dots ,n'\} \inj \{ 0,1,\dots ,n\} $,
there is a corresponding closed immersion 
$i\colon \Delta ^{n'} \inj \Delta ^n$.
An irreducible scheme $V$ over $\Delta ^n$ is said to satisfy the {\it face condition} if for each (ordered) inclusion $i$ as above,
the fiber product $V\times _{\Delta ^n} \Delta ^{n'}$ has dimension $\le \dim (V) - (n-n')$ (when this is the case, the fiber product is purely of dimension $\dim (V) - (n-n')$ or is empty).

Let $t$ be an integer. For each $n\ge 0$, define $z_t (X,n)$ to be the free abelian group 
\[
	z_t(X,n):= \bbZ \left[V\subset X\times \Delta ^n 
	\middle\vert
	\begin{minipage}{6.5cm}
		integral closed subscheme of dimension $n+t$ 
		 satisfying the face condition
	\end{minipage}
	 \right] .
\] 
Thanks to the face condition, the alternating sums of pullbacks along various ordered inclusions $\{ 0,\dots ,n-1\} \inj \{ 0,\dots ,n\} $ give well-defined differential maps
\[
	\cdots \to  z_t(X,n) \to z_t(X,n-1) \to\cdots .
\]
This is Bloch's cycle complex $z_t(X,\bullet )$.

When $t\ge 0$, one can ask if $V$ is equi-dimensional over $\Delta ^n$ (i.e., dominant and the fibers have dimensions $\le t$, or equivalently the non-empty fibers have pure dimension $t$).
One thus obtains a subcomplex $z^{\equi }_t (X,\bullet )\subset z_t(X,\bullet )$.

\subsection{Organization of the notes}
In \S \ref{sec:subdivision-of-simplices} we define the {\it subdivision maps}
$\sd ^\sigma _n\colon X\times \Delta ^n \to X\times \Delta ^n$ for each permutation $\sigma \in \mathfrak S_{\{ 0,1,\dots ,n \} }$,
and the {\it homotopy} maps 
$\sd ^\sigma _{n,k}\colon X\times \Delta ^{n+1} \to X\times \Delta ^n\times \Delta ^1$
for each choice of $0\le k\le n$ and $\sigma \in \mathfrak S _{\{ 0,\dots ,k \} }$.
All these maps depend on the preliminary choice of {\it centers} $c^i\colon X\to \Delta ^i $ for $i\le n$. 

Section \ref{sec:equi-dimensionalization} is the technical key of these notes, where we prove that if the choice of centers is generic, then pullback along subdivision renders a given cycle equi-dimensional.

In \S \ref{sec:proof-Suslin-quis} we complete our proof of Theorem \ref{thm:Suslin-quis} using the subdivision construction.

In \S \ref{sec:modulus} we explain that the exact same proof yields the modulus analog (Theorem \ref{thm:modulus-quis-intro}).

\subsection*{Acknowledgment}
This work grew out of conversations with Ryomei Iwasa and Hiroyasu Miyazaki in 2019.
I did not manage to write it up in a timely manner due to the addition of a new member to my household and a global sanitary crisis that followed. 
In the meantime I have been supported by JSPS KAKENHI grants 
(JP18K13382 and JP22K13886)
and the JSPS Overseas Research Fellowship (the class of 2020).
I also thank the University of Milan for hospitality, where I was a long-term visitor in the years 2022--24 funded by the above-mentioned JSPS fellowship.

\section{Subdivision of simplices (fixing the notation)}\label{sec:subdivision-of-simplices}

Here we set up {\it subdivision maps}
\[
	\Delta ^n \to \Delta ^n
\]
which are determined by the choice of the center and a permutation of $\{ 0,1,\dots , n \} $.
We also have to introduce the homotopy maps 
\[
	\Delta ^{n+1} \to \Delta ^n \times \Delta ^1 .
\]
A similar (and more complicated) construction has been used 
in the proof of the localization property of the higher Chow group
\cite{BlochLoc} \cite{LevineLoc}.

\subsection{Simplices}

Endow the set 
$[n]=\{ 0,\dots ,n \} $ with the order $0<\dots <n$.
Recall 
$\Delta ^n %
:= \Spec ( k [t_0,\dots ,t_n]/(t_0+\dots +t_n -1) )$ is the algebraic $n$-simplex.
Denote by 
$v_i\colon \Spec (k)=\Delta ^0 \to \Delta ^n $ the $i\upperth $ vertex defined by the $k$-algebra map 
$t_j\mapsto 0$ for $j\in \{ 0,\dots ,n \} \setminus \{ i\} $
and 
$t_i\mapsto 1$.
\subsection{The centers}

Let $A$ be a $k$-algebra.
Suppose that we are given a section $c^n\in \Delta   ^n (A)$ for each $n\ge 0$.
They will serve as the center of subdivision.
For a non-empty subset $S\subset [n] $
consisting of $s+1$ elements, 
consider the corresponding embedding $i_S\colon \Delta ^{s}  \hookrightarrow \Delta ^n  $.
Denote by 
\[ 
	c^n_S\in \Delta ^n  (A)
\] 	
the image of $c^s\in \Delta ^s  (A)$ under this embedding.
If $S$ is a singleton $\{ i\} $, then $c_{ \{ i\} }$ equals the structure map $\Spec (A)\to \Spec (k)$ followed by the $i\upperth$ vertex $v_i\colon \Spec (k)\to \Delta ^n$.

\subsection{The subdivision maps}\label{Sec:subdivision}
For a finite set $T$, denote by $\mathfrak{S}_T$ the group of permutations on $T$.
For each permutation $\sigma \in \mathfrak{S}_{[n]}= \mathfrak{S}_{\{ 0,\dots ,n \} } $,
we consider the {\it subdivision} map (a morphism of $A$-schemes determined by the choice of centers $\{ c^i\} _{i\le n}$)
\begin{equation*}
\sd ^\sigma _n \colon \Delta _A ^n\xrightarrow{ } \Delta _A ^n,
 \end{equation*}
defined to be the affine-$A$-linear map sending the vertex
$v_k$ ($0\le k\le n$) to $c^n_{\sigma (\{ 0,\dots ,k \} )}$.
Concretely, this is the restriction of the map between the ambient spaces
\[
	\bbA ^{n+1}_A \to \bbA ^{n+1}_A
\]
corresponding to the $A$-linear map $A^{n+1} \to A^{n+1}$ 
sending $(t_0,\dots ,t_n)$ to $\sum _{k=0}^n t_k c^n_{\sigma (\{ 0,\dots ,k\} )} $.
See Figures \ref{Fig:triangle}, \ref{Fig:tetra}.

\begin{figure}
\begin{minipage}{0.4\textwidth}
\begin{tikzpicture}[scale=1.4]
	\coordinate [label=above:0] (0) at (1,1.7);
	\coordinate [label=left:1] (1) at (0,0);
	\coordinate [label=right:2] (2) at (2,0);
	\draw (0) -- (1) -- (2) -- (0);

	\node [fill=red,inner sep=1pt,label=left:$c_{\{ 0,1 \} }$] (01) at ($ (0)!.55!(1) $) {};
\node [fill=red,inner sep=1pt,label=right:$c_{\{ 0,2 \} }$] (02) at ($ (0)!.55!(2) $) {};
\node [fill=red,inner sep=1pt,label=below:$c_{\{ 1,2 \} }$] (12) at ($ (1)!.55!(2) $) {};

	\node [fill=blue,inner sep=1pt] (012) at (1.1,0.5) {};

	\draw (0) -- (012) -- (1) -- (012) -- (2) -- (012) -- (01) -- (012) -- (02) -- (012) -- (12);

\end{tikzpicture}
\caption{\, Subdivision $n=2$}\label{Fig:triangle}
\end{minipage}
\begin{minipage}{0.4\textwidth}

\begin{tikzpicture}[scale=2]
	\coordinate [label=above:0] (0) at (1,1.6);
	\coordinate [label=left:1] (1) at (0,0);
	\coordinate [label=below:2] (2) at (1.4,-0.4);
	\coordinate [label=right:3] (3) at (2,0.1);
	\draw [thick] (0) -- (1) -- (2) -- (3) -- (0); 
	\draw [thick] (0)-- (2) -- (1) -- (3);

	\node [fill=red,inner sep=1pt,label=left:$c_{\{ 0,1 \} }$] (01) at ($ (0)!.45!(1) $) {};
	\node [fill=red,inner sep=1pt] (02) at ($ (0)!.45!(2) $) {};
	\node [fill=red,inner sep=1pt,label=right:$c_{\{ 0,3 \} }$] (03) at ($ (0)!.45!(3) $) {};
	\node [fill=red,inner sep=1pt,label=below:$c_{\{ 1,2 \} }$] (12) at ($ (1)!.45!(2) $) {};
	\node [fill=red,inner sep=1pt] (13) at ($ (1)!.45!(3) $) {};
	\node [fill=red,inner sep=1pt,label=right:$c_{\{ 2,3 \} }$] (23) at ($ (2)!.45!(3) $) {};

\newcommand{\threeCent}[3]{\node [fill=blue,inner sep=1pt] (#1#2#3) at ($0.23*(#1)+0.37*(#2)+0.4*(#3)$) {}}

	\threeCent{0}{1}{2};
	\threeCent{0}{1}{3};
	\threeCent{0}{2}{3};
	\threeCent{1}{2}{3};

\newcommand{\fourCent}[4]{\node [fill=brown,inner sep=1pt] (#1#2#3#4) at ($0.22*(#1)+0.27*(#2)+0.3*(#3)+0.21*(#4)$) {}}

	\fourCent{0}{1}{2}{3};
	
\newcommand{\triangleSub}[4]{
\draw [#4] (#1) -- (#1#2#3); 
\draw [#4] (#2) -- (#1#2#3); 
\draw [#4] (#3) -- (#1#2#3); 
\draw [#4] (#1#2) -- (#1#2#3); 
\draw [#4] (#1#3) -- (#1#2#3); 
\draw [#4] (#2#3) -- (#1#2#3) }
	
	\triangleSub{0}{1}{2}{thick};
	\triangleSub{0}{1}{3}{dotted};
	\triangleSub{0}{2}{3}{thick};
	\triangleSub{1}{2}{3}{dotted};
	
\newcommand{\tetraSub}[5]{\draw [#5] (#1) -- (#1#2#3#4); 
\draw [#5] (#2) -- (#1#2#3#4); 
\draw [#5] (#3) -- (#1#2#3#4); 
\draw [#5] (#4) -- (#1#2#3#4);
\draw [#5] (#1#2) -- (#1#2#3#4); 
\draw [#5] (#1#3) -- (#1#2#3#4); 
\draw [#5] (#1#4) -- (#1#2#3#4);
\draw [#5] (#2#3) -- (#1#2#3#4);  
\draw [#5] (#2#4) -- (#1#2#3#4);
\draw [#5] (#3#4) -- (#1#2#3#4);
\draw [#5] (#1#2#3) -- (#1#2#3#4);
\draw [#5] (#1#2#4) -- (#1#2#3#4);
\draw [#5] (#1#3#4) -- (#1#2#3#4);
\draw [#5] (#2#3#4) -- (#1#2#3#4);}
	
	\tetraSub{0}{1}{2}{3}{};

 \end{tikzpicture}
\caption{\, Subdivision $n=3$. There are $4!=24$ tetrahedra.}\label{Fig:tetra}
\end{minipage}

\end{figure}
The following relations of maps $\Delta   ^{n-1}\rightrightarrows \Delta   ^n
$ hold:
\begin{align}\label{eq:subdivision-compatibility}
\sd _n^\sigma \circ \partial _i
&= \sd _n^{\sigma \circ (i,i+1)} \circ \partial _i 
&(0\le i\le n-1) , \\ 
\sd _n^\sigma \circ \partial _{n}
& = \partial _{\sigma (n)} \circ \sd _{n-1}^{\partial _{\sigma (n)}\inv \circ \sigma } ,
\notag
 \end{align}
where $\partial _{\sigma (n)}\inv \circ \sigma $ is the permutation $\tau \in \mathfrak S _{[n-1]}$ that is characterized by 
$\partial _{\sigma (n)}\circ \tau = \sigma |_{[n-1]}$.
For each $n,\sigma $, if we choose appropriate linear coordinates on the source and target (both of which are non-canonically isomorphic to $\bbA ^n_A$), the map $\sd _{n}^\sigma $ can be written as a morphism
\begin{equation}\label{Eq:written_as}\begin{array}{ccl}
\bbA ^n_A =\Spec (A[Y_1,\dots ,Y_n])&\to &\bbA ^n_A= \Spec (A[Z_1,\dots ,Z_n])\\[5pt]
\vvector{Y_1 \\ \vdots \\ Y_n} &\mapsto &
\vvector{Z_1 \\ \vdots \\ Z_n}:=
\begin{pmatrix}
C^1_1 & \cdots & C^n_1 \\
&\ddots & \vdots  \\
&& C^n_n
\end{pmatrix}
\vvector{Y_1 \\ \vdots \\ Y_n}
\end{array} \end{equation}
where $C^i_j\in A$ (which also depend on the choice of centers) and the blank entries in the lower half are zero.
For example, we may choose $v_0$ as the origin of the source and the vectors $v_i-v_0$ as the coordinate axes ($1\le i\le n$).
On the target, we take $v_{\sigma (0)}$ as the origin and $v_{\sigma (i)}-v_{\sigma (0)}$ as the axes ($1\le i\le n$).

\subsection{The homotopy}\label{Sec:homotopy}

Given $0\le k \le n$ and a bijection $\sigma $ on $\{ 0,\dots ,k \} $,
we consider the subdivision map (Figure \ref{Fig:homotopy})
\begin{equation*}
\sd _{ n,k }^\sigma \colon \Delta _A  ^{n+1}\xrightarrow{ }\Delta _A  ^n\times _{A}\Delta _A  ^1
 \end{equation*}
which is the affine-linear map sending the vertices
\begin{equation*}
v_0,\ \dots \ ,\ v_k
\qquad \text{ to }\qquad (c_{\{ \sigma (0)\} },v_0),
\dots ,
(c_{\sigma (\{ 0,\dots ,k  \} ) },v_0)
 \end{equation*}
and the vertices
\begin{equation*}
v_{k+1},\ \dots \ ,\ v_{n+1}
\qquad \text{ to }\qquad 
(v_{k },v_1)\
\dots \ ,\
(v_{n },v_1).
 \end{equation*}

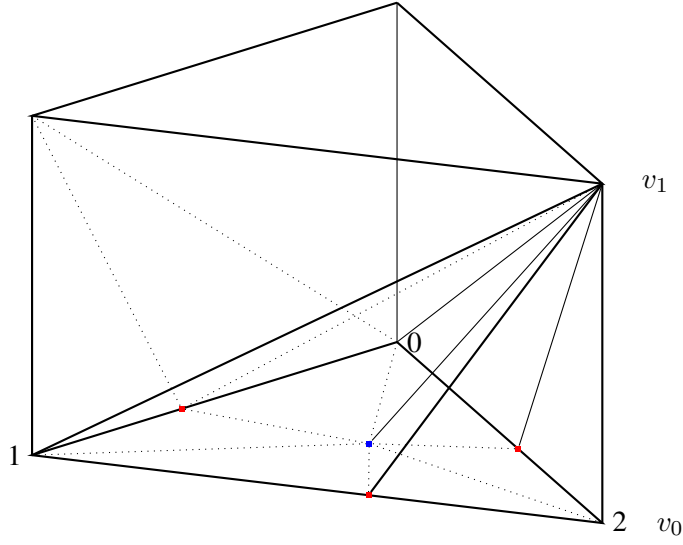
\begin{figure}

\begin{tikzpicture}[scale=3]
	\coordinate [label=right:0] (0) at (1.6,0.5);
	\coordinate [label=left:1] (1) at (0,0);
	\coordinate [label=right:2\quad $v_0$] (2) at (2.5,-0.3);
	\coordinate [] (0') at ($(0)+(0,1.5)$);
	\coordinate [] (1') at ($(1)+(0,1.5)$);
	\coordinate [label=right:\quad $v_1$] (2') at ($(2)+(0,1.5)$);

\newcommand{\drawtriangle}[4]{
\draw [#4] (#1) -- (#2) -- (#3) -- (#1) }

	\drawtriangle{0}{1}{2}{thick};
	\drawtriangle{0'}{1'}{2'}{thick};

	\draw (0) -- (0');
	\draw [thick] (1) -- (1');
	\draw [thick] (2) -- (2');
	
	\draw [dotted] (0) -- (1');
	\draw [] (0) -- (2');	
	\draw [thick] (1) -- (2');	
	
	\newcommand{\lineCent}[2]{
	\node [fill=red,inner sep=1pt] (#1#2) at ($0.41*(#1)+0.59*(#2)$) {}}
	
	\lineCent{0}{1};
	\lineCent{0}{2};
	\lineCent{1}{2};

	\newcommand{\threeCent}[3]{\node [fill=blue,inner sep=1pt] (#1#2#3) at ($0.33*(#1)+0.39*(#2)+0.38*(#3)$) {}}

	\threeCent{0}{1}{2};
	
	\newcommand{\triangleSub}[4]{
\draw [#4] (#1) -- (#1#2#3); 
\draw [#4] (#2) -- (#1#2#3); 
\draw [#4] (#3) -- (#1#2#3); 
\draw [#4] (#1#2) -- (#1#2#3); 
\draw [#4] (#1#3) -- (#1#2#3); 
\draw [#4] (#2#3) -- (#1#2#3) }

	\triangleSub{0}{1}{2}{dotted};

	\newcommand{\lineToTwo}[2]{
\draw [#2] (#1) -- (2')}

	\lineToTwo{01}{dotted};
	\lineToTwo{12}{thick};
	\lineToTwo{02}{};
	\lineToTwo{012}{};
	
	\draw [dotted] (01) -- (1');
 \end{tikzpicture}
\caption{Homotopy for $n=2$. There are $1!+2!+3!=9$ tetrahedra.}\label{Fig:homotopy}

\end{figure}

Let $\pr _1\colon \Delta ^n\times \Delta ^1 \to \Delta ^n$ be the first projection.
The collection of maps 
\begin{equation*}
\pr _1\circ \sd _{n,k}^\sigma \colon \Delta ^{n+1}_A \rightrightarrows \Delta ^n_A
 \end{equation*}
will serve as a homotopy later on.
It takes some patience to verify the following relations of maps $\Delta   ^{n+1}\rightrightarrows \Delta   ^n\times \Delta   ^1$
(which, by the way, imply the relations \eqref{eq:subdivision-compatibility}):
\begin{align}\label{Eq:patience}
\sd _{n,0}^{\id }\circ \partial _0 &= (\id ,\const _{v_1})
\qquad && 
\begin{minipage}{0.3\textwidth}
	(where $\const _{v_1}$ is the constant  
	map $\Delta ^n  \to \Delta ^1  $ to $v_1$); 
\end{minipage}
\\[1pt]
\sd _{n,n}^\sigma \circ \partial _{n+1}&=
(\sd _n^\sigma ,\const _{v_0})
\qquad && (\sigma \in \mathfrak{S}_{[n]}); \notag \\[5pt]
\sd _{n,k}^\sigma \circ \partial _i&=\sd _{n,k}^{\sigma \circ (i,i+1) }\circ \partial _i
\qquad &&(0\le i<k,\ \sigma \in \mathfrak{S}_{[k]}); 
\notag \\[5pt]
\sd _{n,k}^\sigma \circ \partial _i
&= (\partial _{i-1}\times \id _{\Delta   ^1})\circ \sd _{n-1,k}^\sigma 
\qquad &&(k+1< i\le n, \sigma \in \mathfrak{S}_{[k]});
\notag \\[5pt]
\sd _{n,k}^\sigma \circ \partial _{k+1}&= \sd _{n,k+1}^\sigma  %
\circ \partial _{k+1}
\qquad &&
(0\le k< n, \sigma \in \mathfrak{S}_{[k]});
\notag \\[5pt]
\sd _{n,k}^\sigma \circ \partial _k
&= (\partial _{\sigma (k)}\times \id _{\Delta   ^1})\circ 
\sd _{n-1,k-1}^{\partial _{\sigma (k)}\inv \circ \sigma }  %
\qquad &&(1\le k\le n, \sigma \in \mathfrak{S}_{[k]} ),
\notag
\end{align}
where 
$\partial _{\sigma (k)}\inv \circ \sigma  %
\in \mathfrak{S}_{[k-1]}$ 
is the unique element $\tau $ satisfying
$\partial _{\sigma (k)}( \tau (j)) = \sigma (j )$
for all $0\le j\le k-1$.
In the second to last formula, we wrote $\sigma $ also for the permutation on $[k+1]$ acting on the subset $[k]$ by $\sigma $ and fixing $k+1$.

For each $n,k,\sigma $, under appropriate choice of linear coordinates (uniform in the choices of the centers), the morphism $\pr _1\circ \sd _{n,k}^\sigma $ can be written as
\begin{multline}\label{Eq:homotopy_written_as}
\bbA ^{n+1}_A =\Spec (A[Y_1,\dots ,Y_{n+1}])
\quad\to\quad  \bbA ^n_A=\Spec (A[Z_1,\dots ,Z_n]) \\[5pt]
\vvector{Y_1 \\[5pt] \vdots \\[5pt] \vdots \\[5pt] \vdots \\[5pt] Y_{n+1} }
\quad\mapsto\quad
\vvector{Z_1 \\{} \\ \vdots \\ \vdots \\{} \\ Z_{n}}:=
\left( \begin{array}{ccc|c|ccc}
C^1_1 &\cdots & C^k_1 & &&& \\
&\ddots &\vdots &1
&&& \\
&&C^k_k &&&& \\ \hline
&&&&1 && \\
 &&&&&\ddots & \\
 &&&&&& 1
\end{array}\right)
\vvector{Y_1 \\[5pt] \vdots \\[5pt] \vdots \\[5pt] \vdots \\[5pt] Y_{n+1} }
 \end{multline}
%
%
where $C^i_j\in A$ also depend on the centers, the blank entries are zero and the ``$1$'' in the $(k+1)\upperst$ column is placed in the $\sigma\inv (k)\upperth$ row.
An example of such a choice is to take $v_0$ as the origin of the source $\Delta ^{n+1}_A$ and the vectors $v_i-v_0\in \Delta ^{n+1}(A)$ as its coordinate axes ($1\le i\le n+1$); 
on the target $\Delta ^n_A$, take $v_{\sigma (0)}$ as the origin and
the vectors $v_{\sigma (i)}-v_{\sigma (0)}$
($1\le i\le k$), $v_i -v_{\sigma (0)}$ ($k+1\le i\le n $) as the axes.

\section{Equi-dimensionalization}\label{sec:equi-dimensionalization}

We work over an {\em infinite} field $k$.
We set the base ring $A$ to be the polynomial ring $k[X_1,\dots ,X_m]$.
Products ``$\times $'' of schemes will mean the fiber products over $k$.
The main purpose of this section is to show that given a closed subset $V\subset \Spec k[X_1,\dots ,X_m] \times \Delta ^n $
of dimension $\le n+t$, then for a generic choice of the center the inverse image by the subdivision map 
$(\sd ^\sigma _n)\inv (V)\subset \Spec k[X_1,\dots ,X_m] \times \Delta ^n $ has fiber dimensions $\le t$ over $\Delta ^n$ (Corollary \ref{cor:pullback-equi-dimensional}).

\subsection{Affine-linear maps of vector spaces}\label{sec:affine-linear-maps}

Recall that an {\em affine-linear} map $V_1\to V_2$ between vector spaces (over some field) is the composite of a linear map and translation on the target.
An {\em affine-linear subset} of a vector space is the image of an affine-linear map.
It can be written as $\bm{v}+W$ with $\bm{v}$ an element and $W$ a (uniquely determined) linear subspace.
The {\em dimension} and { \em codimension} of an affine-linear subset is defined to be those of $W$.
(In our treatment, it might be more consistent to include the empty subset as an affine-linear subset, but this is of minor importance.)
Suppose an affine-linear map as above has image with dimension $\ge c$. Then the inverse image of any element has codimension $\ge c$ in the source, because it is either empty or an affine-linear subset having codimension equaling the dimension of the image.

\begin{lemma}\label{Lem:enough_codimension}
Let $W\subset $ 
$\bbA ^m\times \bbA ^n =\Spec (k[X_1,\dots ,X_m;Z_1,\dots ,Z_n])$ be an irreducible closed subscheme of dimension $n+t$ (where $t\ge 0$). Consider the following affine-linear map of vector spaces, where the source denotes the vector space of polynomials of degree $\le N$:
\begin{align*}
k[X_1,\dots ,X_m]_{\le N} &\to \Gamma (W,\mcal{O}_W) \\[5pt]
C(\undl{X})&\mapsto Z_1-C(\undl{X}).
 \end{align*}
Then the inverse image of $0\in \Gamma (W,\mcal{O}_W)$ has codimension $\ge N+1$ on the source.
\end{lemma}

\begin{proof}
If the image of $W$ to $\bbA ^m$ is a (necessarily closed) point $x$, then for a dimension reason we see that the support of $W$ is $\bbA ^n_{ k(x) }$ and that the inverse image is empty.

Next, suppose that the image of $W\to \bbA ^m$ has dimension $\ge 1$.
By \S \ref{sec:affine-linear-maps}, it suffices to show that the image of the restriction map
$k[X_1,\dots ,X_m]_{\le N}\to \Gamma (W,\mcal{O}_W)$ has dimension $\ge N+1$.
Since the image of $W\to \bbA ^m$ now has dimension $\ge 1$ and since the base field is infinite, we can find a hyperplane in $\bbA ^m$ not containing the image of $W$.
It follows that there is a linear form $l(\undl{X})\in k[X_1,\dots ,X_m]_1$ mapping to an element $\theta $ in the ring $\Gamma (W,\mcal{O}_W)$ which is transcendental over $k$.
Then the elements $1,l(\undl{X}), \dots ,l(\undl{X})^N \in k[\undl{X}]_{\le N}$
map to
$1,\theta ,\dots ,\theta ^N \in \Gamma (W,\mcal{O}_W)$
which are linearly independent over $k$. 
This completes the proof.
\end{proof}
\subsection{Subdivision with respect to the universal center}\label{Sec:choice_of_centers}
When we later actuate the moving machinery, we will want to choose centers of subdivision $c^i(\undl{X})\in \Delta ^i(k[X_1,\dots ,X_m])$ whose coordinates are represented by polynomials of degree $\le N$  for some integer $N$.
As is always the case with moving lemmas, provided $N$ is large enough, any generic choice of $c^i$ will serve our purposes.
However, to discuss which closed subset our ``generic'' centers have to avoid (and how large $N$ has to be), we must consider the universal situations first.

Let $\bbA _{\le N}$ be the affine space parametrizing polynomials in variables $X_1,\dots ,X_m $ of degree $\le N$. It has dimension $\frac{(N+m)!}{N!m!}$.

Let $p\ge 0$ be an integer and $\Lambda \subset \{ 1,\dots ,p\} \times \{ 1,\dots ,n \} $ be a set of indices. For pairs $(i,j)$ not in $\Lambda $, suppose we are given constant polynomials 
\begin{equation*}
C^i_j(\undl{X}):= C^i_j\in k \quad\text{ for } (i,j)\not\in \Lambda .
\end{equation*}
Given these data,
we can consider the morphism
\begin{align*}
\Sd \colon ( \bbA _{\le N} )^\Lambda \times \bbA ^m\times \bbA ^p
&\to \bbA ^m\times \bbA ^n \\
(\{ C^i_j(\undl{X})\} _{(i,j)\in \Lambda },
\undl{X},\undl{Y})
&\mapsto (\undl{X},\left( \sum _{i=1}^p C^i_j(\undl{X})Y_i \right) _{1\le j\le n}) .
 \end{align*}

When a rational point $\{ C^i_j(\undl{X})\} _{(i,j)\in \Lambda } \in ( \bbA _{\le N} )^\Lambda (k)$ is given, write 
\[ 
\sd _{\{ C^i_j\} }:= \Sd (\{ C^i_j(\undl{X})\} ,-,-)
\colon \bbA ^m\times \bbA ^p\to \bbA ^m\times \bbA ^n .
\]
We saw in \eqref{Sec:subdivision}\eqref{Sec:homotopy}
that under appropriate choice of linear coordinates on simplices $\Delta ^\bullet $, the morphisms $\sd _n^\sigma $ and $\pr _1\circ \sd _{n,k}^\sigma $ are equal to
$\sd _{\{ C^i_j\} }$ for appropriate choices of $\Lambda $ and $ C^i_j(\undl{X})$.
(Recall that we now set $A=k[X_1,\dots ,X_m]$, which amounts to $\bbA ^m$.)

\subsection{Bad locus}

Now suppose we are given a closed subset $V\subset \bbA ^m\times \bbA ^n$ of dimension $\le n+t$.
We want to argue that $\sd _{ \{ C^i_j \}  } \inv (V) \subset \bbA ^m\times \bbA ^p $ has fibers of dimension $\le t$ over $\bbA ^p $ for a generic choice of $\{ C^i_j(\undl{X})\} $.
For given constants $C^i_j\in k$ (for each $(i,j)\not\in \Lambda $), define the ``bad locus'' on which this condition is not true:
\begin{equation}\label{Eq:bad_locus}
B (V) := \left\{ \alpha \in ( \bbA _{\le N} )^\Lambda \times \bbA ^p
\middle|
\begin{array}{c}
	\text{the fiber}\\
	\Sd \inv (V) _\alpha \subset \bbA ^m_{k(\alpha )} \\ 
	\text{has dimension $>t$} \end{array} \right\} 
\subset ( \bbA _{\le N} )^\Lambda \times \bbA ^p .
\end{equation}
It is a constructible subset (a consequence of e.g.\, \cite[IV${}_{1}$, 1.8.4 and IV${}_{3}$, 13.1.3]{EGA4}).
Note that if the first component of $\alpha $ is $\{ C^i_j(\undl{X})\} _{(i,j)\in \Lambda } \in ( \bbA _{\le N} )^\Lambda (k(\alpha ) )$, 
then $\Sd \inv (V) _\alpha = \sd _{ \{ C^i_j \}  } \inv (V\otimes k(\alpha ) ) $.

\subsection{A pointwise argument}\label{Sec:point-wise}
The following assertion is useful when estimating the dimension of a constructible subset (seen as a topological space by the subset topology).

Let $f\colon X\to Y$ be a map of finite-type $k$-schemes and $C\subset X$ be a constructible subset.
If the fiber $C_y$ has codimension $\ge c$ in $X_y$ for all $y\in Y$, then $C$ has codimension $\ge c$ in $X$.

This follows for example from the fact that the dimension of a constructible subset of a finite-type $k$-scheme is the maximum of the transcendental degrees of its points (which is true for locally open subsets hence for constructible subsets of such schemes).
When we apply it, the map $f$ is often a projection $X\times Z\to Z$.

\begin{proposition}\label{Prop:codimension}
Let $V\subset \bbA ^m\times \bbA ^n$ be a closed subset of dimension $\le n+t$ with some $t\ge 0$.
Suppose a $\Lambda \subset \{ 1,\dots ,p \}\times \{ 1,\dots ,n \}$ and $C_{j}^i \in k $ for $(i,j)\not\in \Lambda $ have been chosen and define the constructible subset $B (V) \subset ( \bbA _{\le N} )^\Lambda \times \bbA ^p$ by \eqref{Eq:bad_locus}.
Let $y\in \bbA ^p_k $ be a point and write $(y_1,\dots ,y_p)\in \bbA ^p(k(y))$ for its coordinates.
Assume that either of the following is true:
\begin{itemize}
\item[(i)] for all $1\le j\le n$, there exists $1\le i\le p$ such that $y_i\neq 0$ and $(i,j)\in \Lambda $; or

\item[(ii)] the projection $V\to \bbA ^n$ has fiber dimensions $\le t$.
\end{itemize}
Then the fiber $B (V) _y:= B (V) \times _{\bbA ^p} \{ y \}$ has codimension $\ge N+1$ in $(\bbA _{\le N} )^\Lambda _{k(y)}$.
\end{proposition}

\begin{proof}
Base-change to $k(y)$ does not affect Assumptions (i)(ii) and the conclusion. Thus, making this base-change, we may assume $y$ is a $k$-rational point.

For a point $\alpha =(\{ C^i_j(\undl{X})\} _{i,j}, (y_1,\dots ,y_p))\in ( \bbA _{\le N} )^\Lambda \times \bbA ^p$,
the fiber $\Sd \inv (V)_{\alpha }$ can be obtained by intersecting 
\begin{equation*}
V_{k(\alpha )}\quad\text{ in }\quad \bbA ^m_{k(\alpha )}\times _{k(\alpha )}\bbA ^n_{k(\alpha )} 
=\Spec (k(\alpha )[X_1,\dots ,X_m; Z_1,\dots ,Z_n])
 \end{equation*}
with hypersurfaces
\begin{equation}\label{eq:hypersurfaces_H}
H_j:= \{ Z_j=\sum _{i=1}^p C^i_j(\undl{X})y_i \} 
 \end{equation}
($j=1,\dots ,n$) and projecting it to $\bbA ^m_{k(\alpha )}$ (the projection is then automatically a closed immersion).
Thus our goal is to bound 
$\dim \left( V_{k(\alpha )} \cap \bigcap _{j=1}^n H_j \right) $ 
except when $\{ C^i_j(\underline X) \} _{(i,j)\in \Lambda }$ is in some codimension $\ge N+1$ subset of $(\bbA ^{\le N})^\Lambda $.

Let $J $ be the set of those $j$'s
such that the set of indices 
\begin{equation}\label{Eq:Lambda}
\Lambda (y,j):=
\Lambda \cap (\{ 1\le i\le p \mid y_i\neq 0 \} \times \{ j\} ) 
 \end{equation}
is empty. 
For a given $j$, this is equivalent to the condition that the functions $C^i_j(\undl{X})y_i$ on $(\bbA _{\le N} ) ^\Lambda $ are constants $C^i_jy_i \in k$ for all $1\le i\le p$.
When this is the case, the hypersurface $H_j\subset \bbA ^m_{k(\alpha )}\times _{k(\alpha )}\bbA ^n_{k(\alpha )} $ from \eqref{eq:hypersurfaces_H} is the pullback of a hyperplane $H_{j,0}\subset \bbA ^n $.

Suppose $J$ is non-empty, i.e., Assumption (i) fails, so that Assumption (ii) holds.
Regardless of the choice of $C^i_j(\underline X)\in \bbA _{\le N} $ for
$(i,j)\in \Lambda \cap (\{ 1,\dots ,p\} \times J)$ (all defined over some overfield $k(y)\subset \Omega $), the following intersection of subsets from \eqref{eq:hypersurfaces_H}: 
\[ 
	H_J:= \bigcap _{j\in J} H_j %
\]
is the pull-back of the linear subscheme $H_{J,0}=\bigcap_{j\in J} H_{j,0}\subset \bbA ^n $ of codimension $|J|$. 
So we may compute the intersection $V \cap H_J$
in $\bbA ^m \times  \bbA ^n$ as 
\[
	V\cap H_J = V  \times _{\bbA ^n  } H_0 ,
\]
which has dimension $\le \dim (H_0 )+t $, because 
Assumption (ii) says $V $ has fiber dimensions $\le t$ over $\bbA ^n $.
Since the intersection we are interested in is $V_{k(\alpha )}\cap \bigcap _{j=1}^n H_j = \left( V_{k(\alpha )}\cap H_J \right) \cap \bigcap _{j\not\in J} H_j $,
we can make the following replacements
\begin{align*}
	\{ 1,\dots ,n\} &\leadsto \{ 1,\dots ,n\} \setminus J \\
	\bbA ^m \times \bbA ^n &\leadsto \bbA ^m \times H_{J,0} (\cong \bbA ^m \times \bbA ^{n-|J|} )\\
	V &\leadsto V  \cap H_J  \\ 
	\Lambda &\leadsto \Lambda \setminus (\{ 1,\dots ,p\} \times J )
\end{align*}
to reduce the problem to the case of Assumption (i).
(For this reduction, we used \S \ref{Sec:point-wise} in that in order to bound the codimension of $B(V)_y$ from below, it suffices to bound the codimension of each fiber of $B(V)_y \to (\bbA _{\le N})^{\Lambda \cap (\{ 1,\dots ,p\} \times J )} $ in $(\bbA _{\le N})^{\Lambda \setminus (\{ 1,\dots ,p\} \times J )}$.)

The rest of this proof is an adaptation of a step in the Friedlander--Lawson moving lemma \cite[proof of Prop.~2.2, especially p.107]{FL98}.

From now on we assume Assumption (i) holds.
Let us examine the condition that a tuple $\{ C^i_j(\undl{X})\} _{i,j} \in ( \bbA _{\le N} )^\Lambda (\Omega )$, with $\Omega $ an overfield of $k$, belongs to the bad locus $B(V)_y$. 
Below, we shall often omit the base-change to $\Omega $ from notation.
The condition $\{ C^i_j(\undl{X})\} _{i,j} \in B(V)_y$ implies (though not necessarily equivalent to) that there is $1\le j_0\le n$ such that the following two conditions hold:
\begin{itemize}
\item[(I-${j_0}$)]
the intersection 
$V_{j_0-1}:=
V \cap \bigcap _{j=1}^{j_0 -1} H_j
$ 
in 
$
\bbA ^m \times  \bbA ^n  
$
has dimension $ n+t-(j_0-1)$ (which is necessarily positive); and 
\item[(II-${j_0}$)]
the intersection 
$V_{j_0 -1}\cap H_{j_0} $
still has the same dimension, i.e., $H_{j_0}$ contains an irreducible component of $V_{j_0-1}$.
\end{itemize}
If $V_{j_0-1}=\bigcup _\omega V_{j_0-1}^\omega $ is the decomposition into reduced irreducible components,
Condition (II-${j_0}$) is equivalent to that the function $Z_{j_0}-\sum _{i=1}^p C^i_j(\underline{X})y_i $ restricts to the zero element on one of $V_{j_0-1}^\omega $.

Denote by $B(V)^{j_0}_y \subset (\bbA _{\le N})^\Lambda $ the constructible set of tuples satisfying (I-$j_0$)(II-$j_0$). We have 
\begin{equation*}
B(V)_y\subset \bigsqcup _{j_0=1}^n B(V)^{j_0}_y. \end{equation*}
Hence it suffices to prove that $B(V)^{j_0}_y$ has codimension $\ge N+1$ in $( \bbA _{\le N} )^\Lambda $. 
Note that the conditions (I-$j_0$)(II-$j_0$) do not depend on the choice of $C^i_j(\undl{X})$ for indices $j> j_0$. 
By \S\ref{Sec:point-wise}, it suffices to show that for each choice of $C^i_j(\undl{X})$ ($j<j_0$) satisfying (I-$j_0$), the set of tuples $\{ C^i_{j_0}(\undl{X})\} _{(i,j_0)\in \Lambda (y,j_0)} $ satisfying (II-$j_0$) form a constructible subset of
$( \bbA _{\le N} )^{ \Lambda (y,j_0) }$
having codimension $\ge N+1$,
where the set of indices $\Lambda (y,j_0)$ is as defined in \eqref{Eq:Lambda}.

By Assumption (i), the set $\Lambda (y,j_0)$ is non-empty.
In particular, the affine-linear map
\begin{align*}
( k[X_1,\dots ,X_m]_{\le N} )^{\Lambda (y,j_0)}
&\to k[X_1,\dots ,X_m]_{\le N} \\[5pt]
\{ C^i_{j_0}(\undl{X})\} _{(i,j_0)\in \Lambda (y,j_0)} & \mapsto \sum _{i=1}^p C^i_{j_0}(\undl{X})y_i
 \end{align*}
is surjective.
(It is not necessarily a linear map because $C^i_{j_0}$ with $(i,j_0)\not\in \Lambda $ are prescribed constants, so that the zero element may be mapped to a non-zero constant.)
By this surjectivity and Lemma \ref{Lem:enough_codimension}, the inverse image of zero along the affine-linear map
\begin{align*}
( k[X_1,\dots ,X_m]_{\le N} )^{\Lambda (y,j_0)}& \to \Gamma (V_{j_0-1}^\omega ,\mcal{O}_{V_{j_0-1}^\omega }) \\[2pt]
\{ C^i_{j_0}(\undl{X})\} _{(i,j_0)\in \Lambda (y,j_0)} & \mapsto
Z_{j_0}-\sum _{i=1}^p C^i_{j_0}(\undl{X})y_i
 \end{align*}
has codimension $\ge N+1$ in the source for all $\omega $.
This verifies that $B(V)^{j_0}_y$ has codimension $\ge N+1$ in $( \bbA _{\le N} )^{\Lambda (y,j_0)}$ and completes the proof of Proposition \ref{Prop:codimension}.
\end{proof}

\begin{corollary}\label{cor:pullback-equi-dimensional}

Let $V\subset \bbA ^m\times \Delta ^n $ be a closed subset of pure dimension $n+t$ for some $t\ge 0$, satisfying the face condition with respect to the faces of $\Delta ^n$.
Let $N\ge n+1$.
Then for a generic choice of centers $c^i\in \Hom _{k}(\bbA ^m,\Delta ^i ) $ of degrees $\le N$, $i=1,\dots ,n$, the subdivision maps in \eqref{Sec:subdivision}\eqref{Sec:homotopy} satisfy the following.
\begin{itemize}
\item[(i)]
For all $\sigma \in \mathfrak{S}_{[n]}$, the closed subset $(\sd _n^\sigma )\inv (V)\subset \bbA ^m\times \Delta ^n$ is equi-dimensional over $\Delta ^n$ (of fiber dimension $t$).
\item[(ii)] Suppose $V$ is equi-dimensional over $\Delta ^n$. Then for all $1\le k\le n$ and $\sigma \in \mathfrak{S}_{[k]}$, the closed subset $(\pr _1\circ \sd _{n,k}^\sigma )\inv (V)\subset \bbA ^m\times \Delta ^{n+1}$ is equi-dimensional over $\Delta ^{n+1}$.

\end{itemize}
\end{corollary}

\begin{proof}
It suffices to treat individual $\sigma $ (and $k$) because the conjunction of finitely many conditions which are true for a generic choice is again true for a generic choice.
Denote by $\mcal{A}$ the (finite-dimensional) affine space parametrizing the choice of centers $ c^i\in \Hom _{k}(\bbA ^m,\Delta ^i ) $ of degrees $\le N$ ($1\le i\le n$).
It suffices to show that the bad locus
$B \subset \mcal{A}\times \Delta ^p $ defined as in \eqref{Eq:bad_locus} 
(where $p=n$ or $n+1$) 
has codimension $\ge N+1$, because then its image into $\mcal{A}$ will have codimension $\ge N+1-p \ge 1$.

This being said, (ii) is obviously a particular case of Proposition \ref{Prop:codimension}(ii).

Let us consider (i).
Let $y\in \Delta ^n$ be any point. By the pointwise argument \ref{Sec:point-wise}, it suffices to show $B(V)_y \subset \mcal{A}_y $ has codimension $\ge N$.
Take a coordinate systems on the source and target as in \eqref{Eq:written_as}.
Let $y=(y_1,\dots ,y_n)\in \bbA ^n(k(y))$ (on the source) be the coordinate of $y$ with respect to this coordinate system and let $n_0$ be the greatest index $i$ with $y_i\neq 0$.
We may regard $y$ as the point $(y_1,\dots ,y_{n_0})$ of $\bbA ^{n_0}$ and
our map may be regarded as an endomorphism on
$\bbA ^m\times \bbA ^{n_0} $ given by
\begin{equation*}
( \undl{X}, \vvector{Y_1 \\ \vdots \\ Y_{n_0} }
)
\quad\mapsto\quad
(\undl{X}, \vvector{Z_1 \\ \vdots \\ Z_{n_0} }:=
\begin{pmatrix}
C^1_1(\undl{X}) & \cdots & C^{n_0}_1(\undl{X}) \\
&\ddots & \vdots  \\
&& C^{n_0}_{n_0}(\undl{X})
\end{pmatrix}
\vvector{Y_1 \\ \vdots \\ Y_{n_0} } )
 \end{equation*}
Since this $\bbA ^{n_0}$ corresponds to a face $\Delta ^{n_0}\inj \Delta ^n$, we may replace the problem with that for $V\times _{\Delta ^n}\Delta ^{n_0}$ which falls into Case (i) of Proposition \ref{Prop:codimension} (with $i=n_0 $ fulfilling the assumption for all $1\le j\le n_0$).
\end{proof}

In application,
this Corollary \ref{cor:pullback-equi-dimensional} is supplemented by the following assertion:

\begin{lemma}\label{lem:homotopy-well-defined}
Let $V\subset \bbA ^m\times \Delta ^n$ be a closed subset of pure dimension $ n+t$ for some $t\ge 0$ satisfying the face condition.
Let $N\ge 0$ be any integer.
Then for a generic choice of centers $c^i\in \Hom _k(\bbA ^m,\Delta ^i)$ of degrees $\le N$, $1\le i\le n$, the closed subset
\begin{equation*}
( \sd _{n,k}^\sigma )\inv (V)\subset \bbA ^m\times \Delta ^{n+1}
 \end{equation*}
has dimension $n+t+1$ and satisfies the face condition.
\end{lemma}
\begin{proof}
The following proof is routine but tedious to write down, reflecting the fact that the map $\sd _{n,k}^\sigma $ restricted to the faces can have complicated form.
As in the proof of Corollary \ref{cor:pullback-equi-dimensional}, it suffices to treat individual $k$ and $\sigma $.
Let $\mcal{A}$ be the finite-dimensional affine space parametrizing the choice of centers.
The choice of linear coordinates on $\Delta ^{n}$
gives an isomorphism $\mcal{A}\cong $
$( \bbA _{\le N} )^\Lambda $
with $\Lambda =\{ (i,j)\mid 1\le i\le j\le n \} $.

As in \eqref{Sec:choice_of_centers},
the map $\sd _{n,k}^\sigma $ is the restriction to the point $\{ c^i \} _{1\le i\le n} \in \mcal{A}$ of a map:
\begin{equation*}
\Sd _{n,k}^\sigma \colon \mcal{A}\times \bbA ^m\times \Delta ^{n+1}
\to \bbA ^m\times \Delta ^{n}.
 \end{equation*}
We know that $\Sd _{n,k}^\sigma $ can be written in the form \eqref{Eq:homotopy_written_as}.
Faces of $\Delta ^{n+1}$ and $\Delta ^n$
correspond to equations $\{ Y_i=0\} $ and $\{ Y_1+\dots +Y_{n+1}=1\} $,
and 
$\{ Z_j=0\} $ and $\{ Z_1+\dots +Z_{n}=1\} $
respectively.
From this, restriction of $\Sd _{n,k}^\sigma $ to a face $\mcal{A}\times \bbA ^m\times \Delta ^{p_0}$ of the source has the following form.
\begin{align*}
F\colon ( \bbA _{\le N} )^\Lambda \times \bbA ^m\times \Delta ^{p_0}
&\to
\bbA ^m\times \Delta ^n \\
( 
\{ C^i_j(\undl{X}) \} _{i,j} , \undl{x},\undl{y}
)
&\mapsto 
( \undl{x},
\left( \begin{array}{cc|c}
\bar{C}^1_1(\undl{x}) \cdots & \bar{C}^{k_0}_1(\undl{x}) \\
\vdots &\vdots & * \\
\bar{C}^1_{\lambda (1)}(\undl{x})\ddots & & \\
& \bar{C}^{k_0}_{\lambda (k_0)}(\undl{x}) & \\ \hline
&&I 
\end{array}\right)
\vvector{y_1 \\ \vdots \\ y_{p_0} }
)
 \end{align*}
where:
\begin{itemize}
\item $1\le k_0\le k$ and $1\le \lambda (1)\le \dots \le \lambda (k_0)\le n $ are some integers;
\item $\bar{C}^\alpha _\beta (\undl{X})$ 
are certain (linearly independent) $k$-linear combinations of $C^i_j(\undl{X})$'s;
\item entries denoted by ``$*$'' are some functions on $( \bbA _{\le N} )^\Lambda $;
\item $I$ is a matrix each of whose rows and columns has at most one non-zero entry, and whose non-zero entries are all $1$.
\end{itemize}
Therefore it suffices to prove the following property, which we prove by induction on $p_0$:
\begin{equation}\label{Eq:next_assertion}
\tag{{\bf P}($p_0,n,V$)} 
\begin{minipage}{0.72\textwidth}
		Let $V\subset \bbA ^m\times \Delta ^n$ be a closed subset of dimension $n+t$ ($t\ge 0$) 
	satisfying the face condition.
	Then the inverse image of $V$ along any map $F $ from $( \bbA _{\le N} )^\Lambda \times \bbA ^m\times \Delta ^{p_0}$ as above
	has codimension $\ge m-t $ on the source.
\end{minipage}
\end{equation}
This condition implies that a generic fiber of the projection $F \inv (V)\to ( \bbA _{\le N} )^\Lambda $ has codimension $\ge m-t$ in $\bbA ^m\times \Delta ^{p_0}$ (base-changed to the residue field of the chosen point of $( \bbA _{\le N} )^\Lambda $).
In our original setting, this says that the subset
\begin{equation*}
( \sd _{n,k}^\sigma ) \inv (V) \times _{\Delta ^{n+1} }\Delta ^{p_0} \quad\subset\quad \bbA ^m\times \Delta ^{p_0} 
 \end{equation*}
has codimension $\ge m-t$ for a general $\{ c^i\} \in \mcal{A}$.
This is precisely the face condition for this face. 

Moving on to the proof of Property \ref{Eq:next_assertion},
let $J$ be the set of $j\in \{ 1,\dots ,n \} $ such that the $j\upperth$ row of the above matrix is zero.
Set $n_0:= n-|J|$.
Then the map $F$ lands in the face $\bigcap _{j\in J} \{ Z_j=0 \} \cong \bbA ^m\times \Delta ^{n_0}$.
By the face condition on $V$, the closed set $V\cap \bigcap _{j\in J} \{ Z_j=0 \} $ has dimension $n_0+t $ (or is empty).
Replacing $\Delta ^{n}$ with this face $\Delta ^{n_0}$ and $V$ with $V\cap \bigcap _{j\in J} \{ Z_j=0 \} $, we may assume that the matrix $I$ has no trivial row.
Under this assumption,
the map $F $ is smooth on the locus $\{ Y_{k_0}\neq 0\} $. Therefore, the set $F\inv (V) \cap \{ Y_{k_0}\neq 0 \} $ has codimension $\ge m-t$.
On the other hand, the inverse image
$( F|_{ \{ Y_{k_0}=0\} } )\inv (V)$ has codimension $\ge m-t$ in the locus $\{ Y_{k_0}=0\} $ (so {\em a fortiori} in the entirety of the source) by induction on $p_0$.
This concludes that $F\inv (V)$ has codimension $\ge m-t$. 
\end{proof}
\section{Proof of Suslin's quasi-isomorphism}\label{sec:proof-Suslin-quis}

We have made all the necessary geometric preparation.
Now it is time to complete our alternative proof of  Theorem \ref{thm:Suslin-quis}.
By the pullback/pushforward functoriality of the cycle complexes with respect to field extensions of finite degree, we may assume that the base field $k$ is an {\it infinite} field.
Let $X$ be an affine $k$-scheme of finite type.

\subsection{The complex obtained from the faces of a fixed $[n]$}
\label{sec:faces-of-fixed-n}

To assemble the subdivision maps into a map of cycle complexes, 
we borrow a construction from simpicial techniques.

For an integer $n\ge 0$ we consider the following complex $z_t(X, \bullet \inj [n])$: it is a chain complex concentrated in degrees $[0,n]$ such that in homological degree $0\le s\le n$, we have (recall that $[n]=\{ 0,\dots ,n \}$ has $n+1$ elements)
\begin{equation}
	z_t(X,s \inj [n])
	:= \bigoplus _{S\subset [n],\ |S|=s+1}
	z_t(X,s) .
\end{equation}
Let us denote by $z_t(X,S)$ the summand corresponding to $S$.
For each $S=\{ v_0,\dots ,v_s \} $ and
$0\le k\le s$, we give a differential map from 
$z_t(X,S)$ to $z_t(X,S\setminus \{ v_k \} )$ which is equal to 
$(-1)^k$ times the pullback along the inclusion $S\setminus \{ v_k \} \subset S $.
The following map obtained by taking the sum on each $z_t(X,s)$
\[
	z_t(X,s\inj [n])
	= \bigoplus _{S\subset [n],\ |S|=s+1}
	z_t(X,s) 
	\to 
	z_t(X,s) 
\]
defines a map of complexes $z_t(X,\bullet \inj [n]) \to z_t(X,\bullet )$, which we call the canonical map $\sfcan ^{[n]}$. %

By imposing the equi-dimensionality condition, we can define the subcomplex $z_t^\equi (X,\bullet \inj [n])$, which is carried into 
$z_t^\equi (X,\bullet )$ by $\sfcan ^{[n]}$.

\begin{lemma}\label{eq:general-fact-simplicial}
	The maps 
	\begin{align*} 
	& \sfcan ^{[n]}\colon z_t(X,\bullet \inj [n]) \to z_t(X,\bullet ) \\
	\text{and }\quad 
	&\sfcan ^{[n]}\colon z_t^\equi (X,\bullet \inj [n]) \to z_t^\equi (X,\bullet ) 
	\end{align*} 
	induce isomorphisms on homology groups in degrees $0\le s\le n-1$.
\end{lemma}
\begin{proof}
	This is a general fact on simplicial abelian groups.
	See for example \cite[Cor.~1.3]{FriedlanderSuslin}.
	(Beware the difference of conventions here and there. The set $[n]$ has $n+1$ elements.)
\end{proof}

Let us write $z_t(X,\bullet ) / (\equi )$ and $z_t (X,\bullet \inj [n]) / (\equi )$ for the quotients by the subcomplexes of equi-dimensional cycles.
In order to show Theorem \ref{thm:Suslin-quis} saying that $z_t(X,\bullet )/ (\equi )$ is acyclic, 
by Lemma \ref{eq:general-fact-simplicial}
it suffices to show that 
the map 
\begin{equation}\label{eq:want-to-show-weakly-nullhomotopic} 
	\sfcan ^{[n]}\colon 
	z_t(X,\bullet \inj [n])/ (\equi )
	\to 
	z_t(X,\bullet )/ (\equi ) 
\end{equation}
induces the zero maps on homology.
This will follow if we show that the map \eqref{eq:want-to-show-weakly-nullhomotopic} 
is weakly nullhomotopic.
\footnote{
Recall that a map of chain complexes is said to be {\it weakly nullhomotopic} if it is nullhomotopic on every finitely generated subcomplex (finitely many non-trivial terms, each finitely generated) of the source.}
To this end, fix an arbitrary finitely generated subcomplex
\[  
	z^{\fg}_t(X,\bullet\inj [n]) \subset z_t(X,\bullet\inj [n])	
\]
and let $z^{\fg}_t(X,\bullet\inj [n])/(\equi )$ be its image to 
$z_t(X,\bullet\inj [n])/(\equi )$.

\subsection{The subdivision maps induce a map of complexes}

Keep $n\ge 0$ fixed.
Choose a finitely generated subcomplex $z^{\fg}_t(X,\bullet\inj [n])$. 
By enlarging it slightly, we may assume that its degree $s$ part admits the decomposition 
\[ 
	z^{\fg}_t(X,s\inj [n])
	= 
	\bigoplus _{S\subset [n], |S|=s+1} 
	z^{\fg}_t(X,S),
\]
where $z^{\fg}_t(X,S):= z^{\fg}_t(X,s\inj [n]) \cap z_t(X,S)$.

Choose a closed immersion $X\inj \bbA ^m$.
All  the results in \S \ref{sec:equi-dimensionalization} are valid also for closed subsets in $X\times \Delta ^s $ via the immersion into $\bbA ^m\times \Delta ^s$ because all the maps we constructed are $\bbA ^m$-maps.
Below when we apply results from \S \ref{sec:equi-dimensionalization}, we assume that an integer $N\ge n+1$ has been fixed and the centers 
$c^s\in \Delta ^s (\bbA ^m) (\surj \Delta ^s (X))$ ($0\le s\le n$)
of degree $\le N$
have been chosen appropriately generically,
which is possible because we are working on a fixed $z^{\fg}_t(X,\bullet\inj [n])$.

Lemma \ref{lem:homotopy-well-defined} implies that
for each subset $S\subset [n]$ of order $s+1$,
each $0\le k\le s$ and each permutation $\sigma \in \mathfrak S_{[k]}$,
we have well-defined pullback maps along $\pr _1\circ \sd ^\sigma _{s,k}$ ($0\le k\le s\le n$):
\[
	(\pr _1\circ \sd ^\sigma _{s,k})^*\colon
	z^{\fg}_t(X,S) \to z_t(X,s+1)  ,
\]
which carries the equi-dimensional subgroup into its counterpart by Corollary \ref{cor:pullback-equi-dimensional}(ii).
The relations \eqref{Eq:patience} imply that the sum 
$\sum _{\sigma ,k} (-1)^{\sigma }(-1)^k (\pr _1\circ \sd ^\sigma _{s,k})^*$
for various $s$ and $S$ defines a homotopy between 
$\sfcan ^{[n]}$ and the map $\sd ^{[n]}$ defined by $\sum _{\sigma\in\mathfrak S_{s}}
(-1)^{\sigma } 
(\sd ^\sigma _s )^*\colon z^{\fg }_t (X,S)/ (\equi ) \to z_t(X,s) /(\equi )$
for each $S\subset [n]$:
\[
	\sfcan ^{[n]}\simeq \sd ^{[n]}
	\colon \quad
	z^{\fg }(X,\bullet\inj [n]) /(\equi ) \rightrightarrows z_t(X,\bullet ) / (\equi ) 
	.
\]
Corollary \ref{cor:pullback-equi-dimensional}(i) 
says that the map $\sd ^{[n]}$ is actually the zero map (for a generic choice of centers).
Thus we have constructed a nullhomotopy for 
$\sfcan ^{[n]}$ restricted to $z^{\fg}_t(X,\bullet \inj [n])$.
As the finitely generated subcomplex $z^{\fg }_t(X,\bullet\inj [n])$ can be taken arbitrarily large, 
we conclude that 
$\sfcan ^{[n]}$ itself is weakly nullhomotopic and by Lemma \ref{eq:general-fact-simplicial}
the proof of Theorem \ref{thm:Suslin-quis} is now complete.

\section{The modulus variant}\label{sec:modulus}

We quickly describe what can be said about the {\it modulus} analog of Bloch's cycle complex.
Suppose a finite-type $k$-scheme $X$ is endowed with a Cartier divisor $D$.
Let $\overline{\Delta ^n}$ be the closure of $\Delta ^n$ in $\bbP ^{n+1}$ (recall by definition $\Delta ^n\subset \bbA ^{n+1}\subset \bbP ^{n+1}$), which is isomorphic to $\bbP ^n$.
Let $\overline{\Delta ^n}_\infty $ denote the hyperplane $(\overline{\Delta ^n}\setminus \Delta ^n)_{\reduced }$
in $\overline{\Delta ^n}$.

A closed subset $V\subset (X\setminus |D|)\times \Delta ^n$ is said to satisfy the 
{\it modulus condition}
if, writing $\overline V ^N$ for the normalization of (each irreducible component of) the closure $\overline V$ in $X\times \overline{\Delta ^n}$,
the following inequality of Cartier divisors on $\overline V ^N$ holds:
\[
	D |_{\overline V^N} \ge  \overline{\Delta ^n}_\infty |_{\overline V ^N} .
\]

The advantage of the subdivision construction is the following fact.

\begin{lemma}\label{lem:modulus-preserved}
	Assume $D$ is effective.
	If a closed subset $V\subset (X\setminus D)\times \Delta ^n$ satisfies the modulus condition,
	then so does $(\pr _1\circ \sd ^{\sigma} _{n,k})\inv \subset (X\setminus D)\times \Delta ^{n+1}$
	for every $0\le k\le n$ and $\sigma \in \mathfrak S_{[k]}$
	regardless of the choice of centers.
\end{lemma}
\begin{proof}
	The proof is similar to \cite[Prop.~3.6 or Rem.~3.10]{KaiMiyazaki}. %
	Just use the fact that the morphisms 
	$\pr _1\circ \sd ^\sigma _{n,k} $
	are written by polynomials whose degrees in 
	the coordinates of $\Delta ^\bullet $ are $1$.
\end{proof}

\begin{remark}
	With Suslin's construction $ X\times \Delta ^{n+1} \to X\times \Delta ^n$,
	one can only say that if $V$ satisfies the modulus condition with respect to 
	the divisor $(n+1) D$ (a stronger condition), then the pullback satisfies the modulus condition with respect to $D$
	because Suslin's maps have degrees $n+1$ in the coordinates of $\Delta ^n $, $\Delta ^{n+1}$.
	This is why formerly \cite{KaiMiyazaki} we only managed to show the pro-quasi-isomorphism over the thickenings $mD$ ($m\ge 1$).
\end{remark}

Define the (simplicial) {\it cycle complexes with modulus}
\[
	z^{\equi }_t (X|D,\bullet )
	\subset
	z_t (X|D,\bullet )
	\quad
	(\subset z_t(X\setminus |D|,\bullet ) )
\]
by adding the modulus condition (and the equi-dimensionality condition).
The next theorem improves our former result \cite[Th.~1.2]{KaiMiyazaki} with Miyazaki
in that we do not need the limit over the thickenings of $D$ anymore.

\begin{theorem}\label{thm:Suslin-modulus}
Let $X$ be an affine $k$-scheme equipped with an effective Cartier divisor $D$. 
Then for any $t\ge 0$, the inclusion
\begin{equation*}
z^{\equi }_t (X|D,\bullet )
\inj 
z_t (X|D,\bullet )
 \end{equation*}
is a quasi-isomorphism.
\end{theorem}
\begin{proof}
	Thanks to Lemma \ref{lem:modulus-preserved},
	the proof of Theorem \ref{thm:Suslin-quis} given in 
	\S \ref{sec:proof-Suslin-quis} works verbatim.
\end{proof}

\begin{remark}
	As was mentioned in a footnote in \S \ref{sec:introduction}, 
	Theorem \ref{thm:Suslin-modulus} for the cubically defined
	$z_t(X|D,\bullet )$
	(which is not known to coincide with the simplicial unlike Bloch's higher Chow)
	is not yet available because 
	the author has not been able to (and has no plan of further trying to) write down similar subdivision and {\it homotopy} maps of cubes using only linear polynomials.
\end{remark}

\begin{remark}
Theorem \ref{thm:Suslin-modulus} used the bound $\deg _{Y} (f^{n}_k) \le 1$ (notation from \S \ref{sec:introduction}).
The other bound $\deg _{X}(f^{(n)}_k  ) \le n+1$ (Corollary \ref{cor:pullback-equi-dimensional}) can be used to obtain the following (rather incomplete) result.
Let us explain the motivation behind its formulation first.

In light of the use of complexes 
$(C_*\bbZ ^{\mathcal Q,\bbA ^q })(X)= \bbZ ^{\mathcal Q,\bbA ^q }(X\times \Delta ^\bullet )$
of algebraic cycles 
$V\subset \bbA ^q\times X\times \Delta ^\bullet $
quasi-finite and dominant over $X\times \Delta ^\bullet $
in \cite[\S 12]{FriedlanderSuslin},
it might be worthwhile to consider its modulus analog.

Because in Voevodsky's theory, quasi-isomorphisms of sheaves can be detected at spectra of function fields $X=\Spec F $,
the only relevant case of the moving lemma (Theorem \ref{thm:Suslin-quis}) of Suslin 
here is for the cycle complex $z_0(\bbA ^q_F,\bullet )$ (so by base change, for $z_0(\bbA ^q ,\bullet )$ ).

Now, in more general motive theories, the $\bbA ^1$-invariance should be replaced by 
the $(\bbP ^q,\bbP ^{q-1})$-invariance, 
where the specific meaning of the compactification $\bbA ^q \inj (\bbP ^q,\bbP ^{q-1})$ differs from theory to theory:
it is either 
the log compactification,
the avatar for the cofiber of the Gysin map for $\bbP ^{q-1}\inj \bbP ^q$ on cohomology,
or the scheme $\bbP ^q$ with the divisor $\bbP ^{q-1}$ encoding the fact that they are only allowing tame ramification.
The correct way to see it in our setup should be as the pair $(\bbP ^q,-\bbP ^{q-1})$ where the divisor $-\bbP ^{q-1}$ has a {\it minus sign}
because it has been shown by Miyazaki \cite{MiyazakiCubeInvariance} that the (cubical) cycle complex with modulus $z^*(X|D,\bullet )$ 
is $(\bbP ^q,-\bbP ^{q-1})$-invariant.
Schemes equipped with not necessarily effective Cartier divisors have been studied also in 
\cite{BindaWithout} \cite{KahnMiyazakiTriple}.

Since it is probably not true in whatever modulus category that quasi-isomorphisms between
sheaves can be checked at function fields, 
it seems to the author that we should ask if the following inclusion is a quasi-isomorphism
for affine $X$ and effective $D$:
\begin{multline}\label{eq:we-should-ask}
	z^{\equi }_t( (X,D)\otimes (\bbP ^q , -\bbP^{q-1} ) ,\bullet ) 
	\subset 
	z_t( (X,D)\otimes (\bbP ^q , -\bbP^{q-1} ) ,\bullet ) 
	.
\end{multline}
Here, the tensor product of pairs %
denotes the scheme 
$X\times \bbP ^q$ equipped with the 
divisor 
$D\times \bbP ^q -X\times \bbP ^{q-1}$.
One might think that one could use the $(\bbP ^q,-\bbP ^{q-1})$-invariance on both sides,
but
if $t-q<0$, the equi-dimensional cycle complex $z_{t-q}^\equi (X|D,\bullet )$ does not make sense.
\footnote{
In fact, the point of the presence of $\bbA ^q$ in $C_*\bbZ ^{\mathcal Q,\bbA ^q}$ was to make sense of the non-existent ``$z_{-q}^\equi $''.}

Hoping to show that \eqref{eq:we-should-ask} is a quasi-isomophism,
we deploy the subdivision machinery from \S \ref{sec:equi-dimensionalization}.
In doing so, we first choose an embedding $X\inj \bbA ^{p}$
and embed
$X\times (\bbP ^{q}\setminus \bbP ^{q-1}) \inj \bbA ^{p}\times \bbA ^q =:\bbA ^m$.
As the cycle complexes with modulus in this case are by definition certain complexes of cycles on $(X\setminus D)\times \bbA ^q \times \Delta ^\bullet $, %
the arguments concerning the equi-dimensionality 
goes just fine.

The exact analog of Lemma \ref{lem:modulus-preserved} that the subdivision maps preserve the modulus condition
is not necessarily true because we are concerned with a negative Cartier divisor $-\bbP ^{q-1}$ and because we took it away from our sight when we set up the subdivision machinery.
However, by the fact (Corollary \ref{cor:pullback-equi-dimensional}) that the polynomials describing $\sd ^{\sigma }_{s,k}$
have bounded degrees $\le N$ 
in the coordinates of $\bbP ^q\setminus \bbP ^{q-1}\cong \bbA ^q$
(where we can take $N=n+1$ if we are interested in homology degrees $\le n$),
an argument analogous to \cite[Prop.~3.6]{KaiMiyazaki}
shows that the homotopy maps $(\pr _1\circ \sd ^\sigma _{n,k})^*$ carry the group
$z_t( (X,N D)\otimes (\bbP ^q , -\bbP^{q-1} ) ,S ) $,
which has coefficient $N$ in the modulus, into 
$z_t(  (X,D)\otimes (\bbP ^q , -\bbP^{q-1} ) ,s+1 )$ when $|S|-1=s\le n$.

This establishes isomorphisms of pro-homology groups in each degree
for the inclusion 
\begin{multline}
	\Bigl\{ 
		z_t^{\equi }(  (X,D')\otimes (\bbP ^q , -\bbP^{q-1} ) ,\bullet  
	\Bigr\} _{D\subset D' \text{ thickening} }
	\\
	\xrightarrow{} 
	\Bigl\{
		z_t (  (X,D')\otimes (\bbP ^q , -\bbP^{q-1} ) ,\bullet )
	\Bigr\} _{D\subset D' \text{ thickening} }
	.
\end{multline}
Desirably we should be able to show the quasi-isomorphism for any fixed $D$.
\end{remark}

\end{document}